\documentclass[a4paper,12pt]{article}
\usepackage{graphicx} 

\usepackage{latexsym,amsmath,enumerate,amssymb,amsbsy,amsthm,lscape, hyperref} 
\usepackage {graphicx}
\usepackage{float,color,tikz}
\usepackage{subcaption,   comment}

\usepackage{algorithm}
\usepackage{algpseudocode}
\usepackage{layout}

\usepackage{mathtools}
\newcommand{\rank}{{\rm rank}}

\theoremstyle{plain}
\newtheorem{theorem}{Theorem}[section]
\newtheorem{proposition}[theorem]{Proposition}
\newtheorem{lemma}[theorem]{Lemma}
\newtheorem{corollary}[theorem]{Corollary}

\theoremstyle{definition}

\theoremstyle{remark}
\newtheorem{remark}{Remark}[section]
\newtheorem{example}{Example}[section]

\title{Explicit Determinants of Homogeneous Polynomial Evaluation Matrices and Applications}
\author{Somphong Jitman  and Wannarut Rungrottheera  \footnote{S. Jitman (sjitman@gmail.com) and W. Rungrottheera  (rungrottheera\_w@su.ac.th) are  with the Department of Mathematics, Faculty of Science, Silpakorn University, Nakhon Pathom 73000, THAILAND}}
\date{January 25, 2026}

\begin{document}

\maketitle

\begin{abstract}
    In this work, the determinants of matrices constructed by evaluating homogeneous bivariate polynomials at pairs of vectors are investigated. For a polynomial $p(x,y)=\sum\limits_{i=0}^k \alpha_i x^{k-i}y^i$, an  explicit factorization of the determinant of the associated $n\times n$ evaluation matrix $A_{\mathbf{a},\mathbf{b}}(p(x,y))=\bigl[p(a_r,b_s)\bigr]_{r,s=1}^n$  is presented for all $n \ge k+1$ and for all pairs of  vectors $\mathbf a=(a_1,\dots,a_n)$ and $\mathbf b=(b_1,\dots,b_n)$ of length $n$. In particular, it is proved that $\det (A_{\mathbf{a},\mathbf{b}}(p(x,y)))=0$  when $n \ge k+2$, while in the borderline case $n=k+1$ a  closed formula involving Vandermonde determinants is derived in the vector  sets and the coefficients of $p(x,y)$. Several well-known determinants, including those arising from $(x+y)^k$ and classical quotient forms $\frac{a^k-b^k}{a-b}$ and $\frac{a^k+b^k}{a+b}$, emerge as special cases. We also provide a  discussion for $n \le k$, connecting the problem to symmetric functions and generalized Vandermonde determinants.  Finally, applications such matrices and determinants are provided, including  an explicit formula and equivariance law under linear changes of variables for the sum-form \(p(x,y)=f(x+y)\), and  a non-vanishing bound over finite fields via Schwartz-Zippel lemma.

    
\end{abstract}

 \section{Introduction}

 Polynomial evaluation matrices arise naturally in many areas of mathematics, including
 interpolation theory, approximation theory, coding theory, and structured linear algebra (see, e.g., \cite{BjoPer1970,For2010,MacSlo1977}). 
 Given a polynomial $p(x,y)$ and two sets of evaluation points
 $\mathbf a=(a_1,\dots,a_n)$ and $\mathbf b=(b_1,\dots,b_n)$, one may form the matrix
 \[
 A_{\mathbf a,\mathbf b}(p(x,y))
 =
 \bigl[p(a_r,b_s)\bigr]_{r,s=1}^n,
 \]
 whose algebraic and spectral properties reflect both the structure of the polynomial
 and the geometry of the evaluation points.

 Classical examples include Vandermonde matrices and their variations, whose determinants
 encode polynomial interpolation and root separation:
 \[
 V(x_1,\dots,x_n)=
 \begin{bmatrix}
 	1&1&\cdots&1\\
 	x_1&x_2&\cdots&x_n\\
 	x_1^2&x_2^2&\cdots&x_n^2\\
 	\vdots&\vdots&\ddots&\vdots\\
 	x_1^{n-1}&x_2^{n-1}&\cdots&x_n^{n-1}
 \end{bmatrix}
 \quad \text{and }
 \det V(x_1,\dots,x_n)=\prod\limits_{1\le i<j\le n}(x_j-x_i).
 \]
 More recently, matrices arising from evaluating polynomials of the form $f(x+y)$ or
 $f(xy)$ have appeared in diverse contexts, ranging from additive convolution operators
 to coding theory and random matrix models.
 In particular, determinant formulas for matrices of the form
 $\bigl[f(a_r+b_s)\bigr]$ have been obtained in special cases, often relying on
 ad-hoc arguments or strong structural assumptions on $f$ and the evaluation sets.    In \cite{YS2012}, a family of polynomial evaluation matrices   of the form  $[(a_i+b_j)^k]_{i,j}$ was analyzed , where the determinant was computed via a factorization through Vandermonde matrices.  In particular, for the $n\times n$  matrix
 $C=[(a_i+b_j)^k]_{i,j}$, it  has been been shown that 
 \[
 \det(C)=(-1)^{\binom{k+1}{2}}  \left(\prod\limits_{i=0}^{k}\binom{k}{i}\right) 
 \left(\prod\limits_{1\le i<j\le n}(a_j-a_i)\right)
 \left(\prod\limits_{1\le i<j\le n}(b_j-b_i)\right)
 \]
 when $n=k+1$, whereas $\det(C)=0$ for all $n\ge k+2$. Further extensions and related classes of polynomial evaluation matrices
 were discussed in \cite{PS2022}.
 However, existing results are typically subject to one or more limitations:
 They are restricted to specific polynomial forms,
 they address only the   case $n=k+1$,
 or they do not provide a unified explanation for   determinant 
 when the matrix size exceeds the polynomial degree.

 The aim of this paper is to develop a systematic and unified framework for analyzing determinants of polynomial evaluation matrices arising from homogeneous bivariate polynomials.
 Homogeneity is a crucial structural assumption: it confines the polynomial to a finite-dimensional coefficient space and enables the associated evaluation matrix to admit a factorization through structured Vandermonde-type matrices.
 This factorization serves as the fundamental mechanism driving all of our main results. Rather than focusing on isolated examples, we treat the general homogeneous polynomial
 \[
 p(x,y)=\sum_{i=0}^k \alpha_i x^{k-i}y^i
 \]
 and analyze the corresponding evaluation matrix for all values of $n$ relative to $k$.
 This perspective clarifies when determinants  are identically zero,
 when nontrivial closed forms exist,
 and how these phenomena depend on the support of the coefficient vector $(\alpha_i)$. Let $\mathbb{C}[x,y]$ denote the ring of polynomials in two variables over $\mathbb{C}$. For $k\in\mathbb{N}\cup\{0\}$ set
 \[
 \widehat{\mathbb{C}}_k[x,y]
 :=\Bigl\{ p(x,y)=\sum\limits_{i=0}^{k}\alpha_i x^{k-i}y^{i} \mid  \alpha_i\in\mathbb{C}\Bigr\},
 \]
 the vector space of homogeneous polynomials of total degree $k$.
 We also write
 \[
 \widehat{\mathbb{C}}[x,y]:=\bigcup_{k\ge 0}\widehat{\mathbb{C}}_k[x,y].
 \]
 Given vectors $\boldsymbol a=(a_1,\dots,a_n)$ and $\boldsymbol b=(b_1,\dots,b_n)$ in $\mathbb{C}^n$ and a polynomial $p(x,y)\in  \widehat{\mathbb{C}}[x,y]$,  we focus on the corresponding  \emph{polynomial evaluation matrix} \[
 A_{\boldsymbol a,\boldsymbol b}(p(x,y)):=\left[p(a_r,b_s)\right]_{1\le r,s\le n}
 \]
 and its determinant. The  study of  the determinant of the $n\times n$ matrices
 $C=[(a_i+b_j)^k]_{i,j}$   in  \cite{PS2022} and \cite{YS2012}  can be viewed as  evaluation matrices in the special case, where  the polynomial  $p(x,y) = (x+y)^k$. 
 The concept of polynomial evaluation matrices  can be illustrated in the following example.
 \begin{example} 
 	Let $\mathbf a=(a_1,a_2,a_3,a_4)$ and $\mathbf b=(b_1,b_2,b_3,b_4)$ be vectors in $\mathbb{C}^4$.  
 	Let 
 	$
 	p(x,y) = (x+y)^3 = x^3 + 3x^2y + 3xy^2 + y^3 \in \widehat{\mathbb{C}}[x,y]$. Then 
 	\[
 	A_{\mathbf a,\mathbf b}(p(x,y))=
 	\begin{bmatrix}
 		(a_1+b_1)^3 & (a_1+b_2)^3 & (a_1+b_3)^3 & (a_1+b_4)^3\\
 		(a_2+b_1)^3 & (a_2+b_2)^3 & (a_2+b_3)^3 & (a_2+b_4)^3\\
 		(a_3+b_1)^3 & (a_3+b_2)^3 & (a_3+b_3)^3 & (a_3+b_4)^3\\
 		(a_4+b_1)^3 & (a_4+b_2)^3 & (a_4+b_3)^3 & (a_4+b_4)^3
 	\end{bmatrix}.
 	\]
 	Similarly, for
 	$
 	q(x,y)=x^3+xy^2+y^3 \in \widehat{\mathbb{C}}[x,y]
 	$, we have 
 	
 	\[
 	A_{\mathbf a,\mathbf b}(q(x,y))
 	=
 	\begin{bmatrix}
 		a_1^3 + a_1 b_1^2 + b_1^3 & a_1^3 + a_1 b_2^2 + b_2^3 & a_1^3 + a_1 b_3^2 + b_3^3 & a_1^3 + a_1 b_4^2 + b_4^3\\[2mm]
 		a_2^3 + a_2 b_1^2 + b_1^3 & a_2^3 + a_2 b_2^2 + b_2^3 & a_2^3 + a_2 b_3^2 + b_3^3 & a_2^3 + a_2 b_4^2 + b_4^3\\[2mm]
 		a_3^3 + a_3 b_1^2 + b_1^3 & a_3^3 + a_3 b_2^2 + b_2^3 & a_3^3 + a_3 b_3^2 + b_3^3 & a_3^3 + a_3 b_4^2 + b_4^3\\[2mm]
 		a_4^3 + a_4 b_1^2 + b_1^3 & a_4^3 + a_4 b_2^2 + b_2^3 & a_4^3 + a_4 b_3^2 + b_3^3 & a_4^3 + a_4 b_4^2 + b_4^3
 	\end{bmatrix}.
 	\]
 \end{example}

 In this paper, a systematic and unified study of polynomial evaluation matrices
 associated with homogeneous bivariate polynomials is undertaken, encompassing
 and extending classical Vandermonde and sum-form constructions.
 A general factorization of the evaluation matrix $A_{\mathbf a,\mathbf b}(p(x,y))$
 into structured Vandermonde-type matrices and a coefficient-dependent core matrix
 is established for all homogeneous polynomials of fixed degree~$k$.
 This factorization allows us to prove that
 the determinant is identically zero  whenever $n\ge k+2$.
 In the borderline square case $n=k+1$, explicit determinant formulas are obtained,
 including a closed-form expression for sum-form polynomials of the type
 $p(x,y)=f(x+y)$.
 Equivariance properties under linear changes of variables are also analyzed.
 Further applications are developed over finite fields, where probabilistic
 nonvanishing bounds are derived via the Schwartz--Zippel lemma. A  discussion for $n \le k$, connecting the problem to symmetric functions and generalized Vandermonde determinants is given as well.   The paper is organized as follows:
 In Section \ref{sec2},     the basic factorization of evaluation matrices
 is presented together with   explicit determinant formulas in the square case $n=k+1$ and the
 rectangular case $n \le k$.
 In Section \ref{sec3},   various applications of such matrices and determinants are discussed, including the sum-form $p(x,y)=f(x+y)$,  equivariance under linear
 changes of variables, and behavior over finite fields. 
 Finally, the summary is given in   Section \ref{sec4}.

 \section{Determinants of Polynomial Evaluation Matrices}\label{sec2}

 In this section, a factorization of the polynomial evaluation matrix
 $A_{\mathbf{a},\mathbf{b}}(p(x,y))$ through Vandermonde-type matrices is established.
 This factorization reveals  the behavior of the matrix according
 to the size parameter $n$ relative to the degree $k$: the cases $n<k+1$, $n=k+1$,
 and $n\ge k+2$ exhibit fundamentally different rank  and determinant properties of matrices.
 As immediate consequences, explicit rank bounds are obtained and closed-form
 determinant formulas are derived in the square case $n=k+1$, while several
 special configurations are examined in greater detail. When $n\ge k+2$, the matrix $A_{\mathbf a,\mathbf b}(p(x,y))$ necessarily has rank at most $k+1$ and is therefore singular.
 For the general regime $n< k+1$, the Cauchy--Binet formula together with
 generalized Vandermonde identities is applied to express
 $\det\bigl(A_{\mathbf{a},\mathbf{b}}(p(x,y))\bigr)$ in a structured and computable form.

 \subsection{Factorization of Polynomial Evaluation Matrices}
 Fix $k\ge 0$ and $p(x,y)=\sum\limits_{i=0}^{k}\alpha_i x^{k-i}y^i\in \widehat{\mathbb{C}}_k[x,y]$. Given vectors $\boldsymbol a=(a_1,\dots,a_n)$ and $\boldsymbol b=(b_1,\dots,b_n)$ in $\mathbb{C}^n$, define the $(n\times(k+1))$ Vandermonde-type matrices
 \[
 V_{\mathbf a}^{(k)}:=\left[a_r^{ k-j}\right]_{1\le r\le n,\ 0\le j\le k}
 =\begin{bmatrix}
 	a_1^{k} & a_1^{k-1} & \cdots & a_1^{0}\\
 	a_2^{k} & a_2^{k-1} & \cdots & a_2^{0}\\
 	\vdots  & \vdots    & \ddots & \vdots\\
 	a_n^{k} & a_n^{k-1} & \cdots & a_n^{0}
 \end{bmatrix},\]
 and \[
 W_{\mathbf b}^{(k)}:=\left[b_s^{ j}\right]_{1\le s\le n,\ 0\le j\le k}
 =\begin{bmatrix}
 	b_1^{0} & b_1^{1} & \cdots & b_1^{k}\\
 	b_2^{0} & b_2^{1} & \cdots & b_2^{k}\\
 	\vdots  & \vdots  & \ddots & \vdots\\
 	b_n^{0} & b_n^{1} & \cdots & b_n^{k}
 \end{bmatrix}.
 \]
 By expanding $p(a_r,b_s)$ we obtain the matrix identity
 \begin{equation}\label{eq:factor}
 	A_{\boldsymbol a,\boldsymbol b}(p(x,y))
 	=V_{\boldsymbol a}^{(k)} D_\alpha \left(W_{\boldsymbol b}^{(k)}\right)^{ T},
 \end{equation}
 where $D_\alpha=\mathrm{diag}(\alpha_0,\alpha_1,\dots,\alpha_k)$ is given from the coefficients of $p(x,y)$.
 \begin{example} 
 	For $k=3$ and $n=3$,  let   $p(x,y)=(x+y)^3=x^3 + 3x^2y + 3xy^2 + y^3$   and let   $\boldsymbol a=(a_1,a_2,a_3)$ and $\boldsymbol b=(b_1,b_2,b_3)$ in $\mathbb{C}^3$. Then  $D_\alpha=\mathrm{diag}(1,3,3,1)$, 
 	\[
 	V_{\mathbf a}^{(3)}=
 	\begin{bmatrix}
 		a_1^{3} & a_1^{2} & a_1^{1} & a_1^{0}\\
 		a_2^{3} & a_2^{2} & a_2^{1} & a_2^{0}\\
 		a_3^{3} & a_3^{2} & a_3^{1} & a_3^{0} 
 	\end{bmatrix}, \text{~ and ~}
 	W_{\mathbf b}^{(3)}=
 	\begin{bmatrix}
 		b_1^{0} & b_1^{1} & b_1^{2} & b_1^{3}\\
 		b_2^{0} & b_2^{1} & b_2^{2} & b_2^{3}\\
 		b_3^{0} & b_3^{1} & b_3^{2} & b_3^{3} 
 	\end{bmatrix}.
 	\]
 	Hence, 
 	\[
 	A_{\mathbf a,\mathbf b}(p(x,y))=V_{\mathbf a}^{(3)} 
 	\mathrm{diag}(1,3,3,1) \left(W_{\mathbf b}^{(3)}\right)^T.
 	\]
 	
 \end{example}

 \begin{proposition} \label{prop:rank}
 	Let $k\geq 0$   and  $n\geq 1$ be integers and let $p(x,y)\in \widehat{\mathbb{C}}_k[x,y]$.  Then 
 	\[
 	\rank( A_{\boldsymbol a,\boldsymbol b}(p(x,y))) \le k+1.
 	\]
 	In particular, if $n\ge k+2$, then $\det (A_{\boldsymbol a,\boldsymbol b}(p(x,y)))=0$.
 \end{proposition}
 
 \begin{proof}
 	The factorization \eqref{eq:factor} expresses $A$ as a product 
 	\[A_{\boldsymbol a,\boldsymbol b}(p(x,y))
 	=V_{\boldsymbol a}^{(k)} D_\alpha \left(W_{\boldsymbol b}^{(k)}\right)^{ T}, \]
 	where  $V_{\boldsymbol a}^{(k)}$, $D_\alpha$, and $\left(W_{\boldsymbol b}^{(k)}\right)^{ T}$ are    an $n\times (k+1)$ matrix, a $(k+1)\times (k+1)$ matrix,  and a $(k+1)\times n$ matrix, respectively.  By the rank inequality for matrix products,  it follows that   \[\rank(A_{\mathbf a,\mathbf b}(p(x,y)) ) \leq  \min \{ \rank(V_{\boldsymbol a}^{(k)}), \rank(D_\alpha), \rank(\left(W_{\boldsymbol b}^{(k)}\right)^{ T}) \} \leq k+1.\] If $n\ge k+2$, an $n\times n$ matrix of rank  less than or equal to $ k+1$ is always singular.
 \end{proof}
 
 \subsection{The Borderline Case $n=k+1$}
 
 In this subsection,  we focus on the case where
 $n=k+1$.  From \eqref{eq:factor},  
 \[A_{\boldsymbol a,\boldsymbol b}(p(x,y))
 =V_{\boldsymbol a}^{(k)} D_\alpha \left(W_{\boldsymbol b}^{(k)}\right)^{ T},\] which becomes a product of three square matrices. The determinants can then be computed explicitly in terms of Vandermonde products.
 
 \begin{theorem}\label{thm:border}
 	Let $k\ge 0$ and  $n=k+1$ be integers and  let $p(x,y)=\sum\limits_{i=0}^{k}\alpha_i x^{k-i}y^i\in\widehat{\mathbb{C}}_k[x,y]$.
 	Then
 	\[
 	\det (A_{\boldsymbol a,\boldsymbol b}(p(x,y)))
 	= (-1)^{\binom{k+1}{2}} \left(\prod\limits_{i=0}^{k}\alpha_i\right)
 	\left(\prod\limits_{1\le i<j\le n}(a_j-a_i)\right)
 	\left(\prod\limits_{1\le i<j\le n}(b_j-b_i)\right), 
 	\]
 	where $\boldsymbol a=(a_1,\dots,a_n)$ and $\boldsymbol b=(b_1,\dots,b_n)$ in $\mathbb{C}^n$. 
 	In particular, $\det (A_{\boldsymbol a,\boldsymbol b}(p(x,y)))\neq 0$ if and only if
 	$\alpha_i\ne 0$ for all $i$ and the vector sets $\{a_i\}$ and $\{b_i\}$ are pairwise distinct.
 \end{theorem}
 
 \begin{proof}
 	From \eqref{eq:factor}, we have 
 	\[
 	\det (A)
 	=\det\left(V_{\boldsymbol a}^{(k)}\right)\cdot \det(D_\alpha)\cdot \det\left(W_{\boldsymbol b}^{(k)}\right)^T,
 	\] where $D_\alpha=\mathrm{diag}(\alpha_0,\alpha_1,\dots,\alpha_k)$.
 	We note that 
 	$W_{\boldsymbol b}^{(k)}$ is the standard $(k+1)\times(k+1)$ Vandermonde matrix which implies that \[\det(W_{\boldsymbol b}^{(k)})=\prod\limits_{1\le i<j\le n}(b_j-b_i).\]  The matrix $V_{\boldsymbol a}^{(k)}$ has columns in reversing order of the  standard $(k+1)\times(k+1)$ Vandermonde matrix, reversing the column order to the standard order
 	requires $\binom{k+1}{2}$ adjacent transpositions, and  hence

 	\[
 	\det\left(V_{\boldsymbol a}^{(k)}\right)
 	=(-1)^{\binom{k+1}{2}}\prod\limits_{1\le i<j\le n}(a_j-a_i).
 	\]
 	Finally, $\det(D_\alpha)=\prod\limits_{i=0}^{k}\alpha_i$. Multiplying the three determinants gives the claim.
 \end{proof}
 
 \begin{remark} \label{rem:2.1}
 	For  $c\in\mathbb{C}$ and  $p(x,y)=c\cdot q(x,y)$ with $q(x,y)\in \widehat{\mathbb{C}}_k[x,y]$, then
 	$D_{\alpha}=c D_{\beta}$, where   $\alpha$ and  $\beta$ are the vector coefficients of  $p(x,y)$ and $q(x,y)$, respectively. Hence,  $\det (A_{\boldsymbol a,\boldsymbol b}(p(x,y)))=c^{n}\det( A_{\boldsymbol a,\boldsymbol b}(q(x,y)))$.
 	In particular, the determinant is {multi-linear} in the columns of $D_\alpha$ but {not} linear in $p(x,y)$ as a polynomial.
 \end{remark}

 In each  of the following special cases,  we use Theorem~\ref{thm:border} when $n=k+1$, and apply Proposition~\ref{prop:rank} for the vanishing when $n\ge k+2$.

 \begin{corollary}[{\cite[Propositions 1 and 2]{YS2012}}]
 	\label{thm:sum-powers}
 	Let $k\ge 0$ and $n\ge 1$ be integers and let  $p(x,y)=(x+y)^k=\sum\limits_{i=0}^k\binom{k}{i}x^{k-i}y^i$.  Let $\mathbf a=(a_1,\dots,a_n)$  and  $\mathbf b=(b_1,\dots,b_n)$ in $\mathbb{C}^n$. Then the following statements hold. 
 	\begin{enumerate}
 		\item If $n\ge k+2$, then $\det (A_{\mathbf a,\mathbf b}(p(x,y)))=0$.
 		\item If $n=k+1$, then
 		\[
 		\det (A_{\mathbf a,\mathbf b}(p(x,y))
 		)=
 		(-1)^{\binom{k+1}{2}}
 		\left(\prod\limits_{i=0}^k \binom{k}{i}\right)
 		\left(\prod\limits_{1\le i<j\le n}(a_j-a_i)\right)
 		\left(\prod\limits_{1\le i<j\le n}(b_j-b_i)\right).
 		\]
 		In particular, $\det (A_{\mathbf a,\mathbf b}(p(x,y)))\ne 0$ if and only if  the $a_i$'s are pairwise distinct and the $b_j$'s are pairwise distinct.
 	\end{enumerate}
 \end{corollary}
 
 \begin{proof} We note that  each  $i\in \{0,1,\dots, k\}$, we have $\alpha_i=\binom{k}{i}$ which implies that $\det(D_\alpha)=\prod\limits_{i=0}^k\binom{k}{i}$. The result follows from  Theorem~\ref{thm:border}. 
 \end{proof}

 \begin{corollary}\label{thm:all-ones}
 	Let $k\ge 0$ and $n\ge 1$ be integers and let $p(x,y)=\sum\limits_{i=0}^k x^{k-i}y^i$. Let $\mathbf a=(a_1,\dots,a_n)$  and  $\mathbf b=(b_1,\dots,b_n)$ in $\mathbb{C}^n$. Then the following statements hold. 
 	\begin{enumerate}
 		\item If $n\ge k+2$, then $\det( A_{\mathbf a,\mathbf b}(p(x,y)))=0$.
 		\item If $n=k+1$, then
 		\[
 		\det( A_{\mathbf a,\mathbf b}(p(x,y)))
 		=
 		(-1)^{\binom{k+1}{2}}
 		\left(\prod\limits_{1\le i<j\le n}(a_j-a_i)\right)
 		\left(\prod\limits_{1\le i<j\le n}(b_j-b_i)\right).
 		\]

 		Moreover,  we have  the identity
 		\[
 		p(a_r,b_s)=\frac{a_r^{k+1}-b_s^{k+1}}{a_r-b_s}\qquad (a_r\ne b_s),
 		\]
 		which implies that  $A_{\mathbf a,\mathbf b}(p(x,y))$ agrees   with
 		$\left[\dfrac{a_r^{k+1}-(-b_s)^{k+1}}{a_r+b_s}\right]$ on the domain $a_r\ne b_s$.
 	\end{enumerate}
 \end{corollary}
 
 \begin{proof}
 	We note that  each  $i\in \{0,1,\dots, k\}$, we have $\alpha_i=1$ which implies that $\det(D_\alpha)=1$. The result follows from  Theorem~\ref{thm:border}. 
 \end{proof}

 \begin{corollary}\label{thm:alternating}
 	Let $k\ge 0$ and $n\ge 1$ be integers and let $p(x,y)=\sum\limits_{i=0}^k (-1)^i x^{k-i}y^i$.  Let $\mathbf a=(a_1,\dots,a_n)$  and  $\mathbf b=(b_1,\dots,b_n)$ in $\mathbb{C}^n$. Then the following statements hold.
 	\begin{enumerate} 
 		\item If $n\ge k+2$, then $\det (A_{\mathbf a,\mathbf b}(p(x,y)))=0$.
 		\item If $n=k+1$, then
 		\[
 		\det (A_{\mathbf a,\mathbf b}(p(x,y)))
 		=
 		\left(\prod\limits_{1\le i<j\le n}(a_j-a_i)\right)
 		\left(\prod\limits_{1\le i<j\le n}(b_j-b_i)\right).
 		\]
 	\end{enumerate}
 	Moreover, when $k$ is odd one has the identity
 	\[
 	p(a_r,b_s)=\frac{a_r^{k+1}-(-b_s)^{k+1}}{a_r+b_s}\qquad (a_r\ne -b_s),
 	\]
 	which implies that $A_{\mathbf a,\mathbf b}(p(x,y))$ agrees with 
 	$\left[\dfrac{a_r^{k+1}-(-b_s)^{k+1}}{a_r+b_s}\right]$ on the domain $a_r\ne -b_s$.
 \end{corollary}
 
 \begin{proof}
 	We note that  each  $i\in \{0,1,\dots, k\}$, we have $\alpha_i=(-1)^i$ which implies that $\det(D_\alpha)=\prod\limits_{i=0}^k(-1)^i=(-1)^{\binom{k+1}{2}}$ which cancels with the same power from $\det(V_{\mathbf a}^{(k)})$ to give the stated sign. The result follows from  Theorem~\ref{thm:border}. 
 \end{proof}

 \subsection{The Case $n\le k$}
 In this subsection, we focus on the case where $n\le k$. However, the results hold true for  $n= k+1$ as well. Therefore, the results are written in terms of $n\leq k+1$ if there is no ambiguity. 
 
 For an $m\times n$ matrix  $M$ and an  index set $I\subseteq \{1,2,\dots, n\}$ of increasing order, let $M(:,I)$ denote the submatrix of $M$ formed by the columns with indices   in $I$.
 
 \begin{theorem}[{\cite[Theorem 3.A.5 (Cauchy-Binet Formula)]{Johnston2021}}]
 	\label{thm:CB-general} Let $m\leq n $ be positive integers. Let $A$ and $B$ be $m\times n$ and $n \times m$ matrices, respectively. 
 	Then 
 	\[\det(AB)= \sum\limits_{I=\{1\leq i_1<i_2<\dots<i_m\leq n\} } \det(A(:,I))\det(B(:,I)).\]
 	
 \end{theorem}

 From \eqref{eq:factor}, we recall the factorization $A_{\mathbf a,\mathbf b}(p(x,y))=V_{\mathbf a}^{(k)}D_\alpha(W_{\mathbf b}^{(k)})^{ T}$ with
 \[
 V_{\mathbf a}^{(k)}=\left[a_r^{ k-j}\right]_{1\le r\le n, 0\le j\le k},
 \quad
 W_{\mathbf b}^{(k)}=\left[b_s^{ j}\right]_{1\le s\le n, 0\le j\le k},
 \text{ and }
 D_\alpha=\mathrm{diag}(\alpha_0,\dots,\alpha_k).
 \]
 
 Assume that 
 $n\le k+1$.  It follows that $V_{\mathbf a}^{(k)}$ and $W_{\mathbf b}^{(k)}$ are rectangular with the number of   columns is greater than or equal to the number of  rows.
 Applying the Cauchy-Binet formula in Theorem \ref{thm:CB-general} to the product
 $V_{\mathbf a}^{(k)}D_\alpha\left(W_{\mathbf b}^{(k)}\right)^{ T}$, we have 
 \begin{equation}\label{eq:CB-master}
 	\det (A_{\mathbf a,\mathbf b}(p(x,y)))
 	=
 	\sum\limits_{I=\{0\leq i_1<i_2<\dots<i_n\leq k\} }
 	\det \left(V_{\mathbf a}^{(k)}(:,I)\right) 
 	\left(\prod\limits_{i\in I}\alpha_i\right) 
 	\det \left(W_{\mathbf b}^{(k)}(:,I)\right).
 \end{equation}

 For $m\in\mathbb Z$ and variables $\mathbf x=(x_1,\dots,x_n)$ define
 \[
 H_0(\mathbf x)=1,\qquad H_m(\mathbf x)=0\ \ (m<0),
 \]
 and for $m\ge 1$ recursively
 \begin{align} \label{eq:H}
 	H_m(x_1,\dots,x_n)=H_m(x_1,\dots,x_{n-1})+x_n H_{m-1}(x_1,\dots,x_n).
 \end{align}

 \begin{proposition}[{\cite[Page 41 (3.6)]{Macdonald1995}}]\label{prop:genVand-H}
 	For $e_1>\cdots>e_n\ge 0$,
 	\[
 	\det[a_r^{ e_j}]_{r,j=1}^n
 	=
 	\prod\limits_{1\le r<r'\le n}(a_{r'}-a_r) \det\left[ H_{ e_j-(n-j)}(\mathbf a) \right]_{j=1}^n.
 	\]
 \end{proposition}

 \begin{theorem}\label{thm:SCB-RouteB}
 	Let $1\le n\le k+1$  be integers and  let $p(x,y)=\sum\limits_{i=0}^k \alpha_i x^{ k-i}y^{ i}$. Then
 	\begin{align*}
 		\det (A_{\mathbf a,\mathbf b}(p(x,y)))
 		=&
 		\prod\limits_{1\le r<r'\le n}(a_{r'}-a_r) \prod\limits_{1\le r<r'\le n}(b_{r'}-b_r)
 		\\
 		&\sum\limits_{I=\{0\leq i_1<i_2<\dots<i_n\leq k\} }
 		\left(\prod\limits_{i\in I}\alpha_i\right) 
 		\det \left[ H_{ k-i_j-(n-j)}(\mathbf a) \right]_{j=1}^n 
 		\det \left[ H_{ i_j-(n-j)}(\mathbf b) \right]_{j=1}^n,
 	\end{align*}
 	where   $H_m$ are as in \eqref{eq:H}. 
 \end{theorem}
 
 \begin{proof}
 	By  \eqref{eq:CB-master},  it follows that
 	\[
 	\det (A_{\mathbf a,\mathbf b}(p(x,y)))=
 	\sum\limits_{I=\{0\leq i_1<i_2<\dots<i_n\leq k\} }\det\left(V_{\mathbf a}^{(k)}(:,I)\right)\left(\prod\limits_{i\in I}\alpha_i\right)\det\left(W_{\mathbf b}^{(k)}(:,I)\right),
 	\]
 	where the sum runs over $n$–subsets  $I=\{0\leq i_1<i_2<\dots<i_n\leq k\}$. 
 	For each $j\in \{1,2,\dots, n\}$,   let $e_j=k-i_j$. 
 	Applying Proposition~\ref{prop:genVand-H} to the exponent $e_j$ in
 	$ V_{\mathbf a}^{(k)}(:,I) $ yields
 	\begin{align*}
 		\det(V_{\mathbf a}^{(k)}(:,I))&=\prod\limits_{1\le r<r'\le n}(a_{r'}-a_r) \det \left(\left[ H_{ e_j-(n-j)}(\mathbf a) \right]_{j=1}^n \right)\\
 		&
 		=\prod\limits_{1\le r<r'\le n}(a_{r'}-a_r)  \det \left(\left[ H_{ k-i_j-(n-j)}(\mathbf a) \right]_{j=1}^n\right).
 	\end{align*}
 	Similarly, applying Proposition~\ref{prop:genVand-H} to the exponents $i_j$  in $W_{\mathbf b}^{(k)}(:,I) $ gives
 	\[
 	\det(W_{\mathbf b}^{(k)}(:,I))= \prod\limits_{1\le r<r'\le n}(b_{r'}-b_r)  \det \left[ H_{ i_j-(n-j)}(\mathbf b) \right]_{j=1}^n.
 	\]
 	Combine the results,  the theorem follows. 
 \end{proof}
 
 \begin{corollary} \label{cor:rank-vanish} Let $1\le n\le k+1$  be integers and  let $p(x,y)=\sum\limits_{i=0}^k \alpha_i x^{ k-i}y^{ i}$.
 	Let $S=\{i \in \{0,1,\dots, k\} \mid \alpha_i\ne 0\}$. Then
 	\[
 	\rank(A_{\mathbf a,\mathbf b}(p(x,y)))\ \le\ \min\{n, |S|\}.
 	\]
 	If $|S|<n$, then $\det (A_{\mathbf a,\mathbf b}(p(x,y)))=0$.
 \end{corollary}
 
 \begin{proof} Since \(\alpha_i=0\) for \(i\notin S\), we may delete the zero columns and obtain the reduced factorization
 	\[
 	A_{\mathbf a,\mathbf b}(p(x,y))=V_{\mathbf a}^{(k)} D_\alpha (W_{\mathbf b}^{(k)})^{ T} 
 	=
 	V_{\mathbf a}^{(k)}(:,S)  D_\alpha(S)  \left(W_{\mathbf b}^{(k)}(:,S)\right)^{ T},
 	\]
 	where   \(D_\alpha=\mathrm{diag}(\alpha_0,\dots,\alpha_k)\),   \(V_{\mathbf a}^{(k)}(:,S)\),  \(W_{\mathbf b}^{(k)}(:,S)\) are \(n\times |S|\), and
 	\(D_\alpha(S)\) is \(|S|\times |S|\) diagonal with nonzero diagonal entries.
 	Hence,  
 	\[
 	\rank (A_{\mathbf a,\mathbf b}(p(x,y)))
 	\le 
 	\min \{\rank( V_{\mathbf a}^{(k)}(:,S)), \rank (D_\alpha(S)), \rank (W_{\mathbf b}^{(k)}(:,S))\}
 	\le  |S|.
 	\]
 	Clearly,  \(\rank (A_{\mathbf a,\mathbf b}(p(x,y)))\le n\) which implies that 
 	\(\rank (A_{\mathbf a,\mathbf b}(p(x,y))\le \min\{n,|S|\})\).
 	If \(|S|<n\), then \(\rank (A_{\mathbf a,\mathbf b}(p(x,y)))<n\), and hence
 	\(\det (A_{\mathbf a,\mathbf b}(p(x,y)))=0\).
 \end{proof}

 \section{Applications}
 \label{sec3}
 
 In this section, we explore several consequences and applications of the determinant factorization.  
 We begin with a closed-form  formula for the determinant of polynomial evaluation matrices in  the special case $p(x,y)=f(x+y)$,  
 followed by equivariance   under linear changes of variables.  
 We then analyze the behavior of the determinant over finite fields, providing probabilistic bounds on when it is identically zero.

 \subsection{Determinant Formula for  {$p(x,y)=f(x+y)$} }
 \label{subsec:closed-form-sum}
 
 We begin with the special case where the bivariate polynomial is of \emph{sum-form},  
 \[
 p(x,y)=f(x+y),
 \]
 with \(f(t)\) a univariate polynomial of degree \(k\) and the matrix size \(n=k+1\).  
 In this setting, the determinant of the polynomial evaluation matrix admits a clean factorization in terms of the Vandermonde factors of the evaluation points.  
 This gives an explicit and computationally convenient closed-form formula for \(\det (A_{\mathbf a,\mathbf b}(f(x+y)))\),  
 revealing its full algebraic structures.

 \begin{proposition} \label{prop:resultant-form} Let 
 	$n=k+1$  be integers and let $\mathbf a=(a_1,\dots,a_n)$  and  $\mathbf b=(b_1,\dots,b_n)$ be in $\mathbb{C}^n$.
 	If $f(t)=\sum\limits_{i=0}^k \alpha_i t^i$ with leading coefficient $\alpha_k\neq 0$, then
 	\[
 	\det \left(A_{\mathbf a,\mathbf b}(f(x+y))\right)
 	=  \alpha_k^{ n} (-1)^{\binom{n}{2}}
 	\prod\limits_{i=0}^k \binom{k}{i}\prod_{1\leq r<r' \leq n}(a_{r'}-a_r) \prod_{1\leq s<s' \leq n}(b_{s'}-b_s).\]
 \end{proposition}
 
 \begin{proof}
 	Let $V_{\mathbf a}^{(k)}=[ a_r^{ k-j} ]_{1\le r\le n, 0\le j\le k}$ and
 	$W_{\mathbf b}^{(k)}=[ b_s^{ j} ]_{1\le s\le n, 0\le j\le k}$.
 	Write also the “ascending–power” evaluation matrices
 	\[
 	E_{\mathbf a}=[ a_r^{ i-1} ]_{1\le r,i\le n},\qquad
 	E_{\mathbf b}=[ b_s^{ j-1} ]_{1\le s,j\le n}.
 	\]
 	Let $R$ be the $(k+1)\times(k+1)$ column-reversal permutation matrix (ones on the anti-diagonal, zeros elsewhere). Then
 	\[
 	E_{\mathbf a} = V_{\mathbf a}^{(k)} R,
 	\qquad
 	E_{\mathbf b} = W_{\mathbf b}^{(k)}.
 	\]
 	
 	By definition of $f(t)$ and the binomial theorem, we have 
 	\[
 	f(a_r+b_s)=\sum_{m=0}^k \alpha_m(a_r+b_s)^m 
 	= \sum_{m=0}^k \alpha_m \sum_{i=0}^m \binom{m}{i} a_r^i b_s^{ m-i}.
 	\]
 	Set $j=m-i$, so $i+j=m$. The coefficient of $a_r^i b_s^j$ is $\alpha_{i+j}\binom{i+j}{i}$, so
 	\[
 	f(a_r+b_s)=\sum_{i=0}^{n-1}\sum_{j=0}^{n-1} \alpha_{i+j}\binom{i+j}{i} a_r^i b_s^j.
 	\]
 	Restricting to $i,j\le n-1$ ensures all terms fit into $n\times n$ matrices (and if $n=k+1$ all needed terms appear).
 	
 	Now write this in matrix form: the $i$th power $a_r^i$ gives the $(r,i+1)$ entry of $E_{\mathbf a}$, similarly $b_s^j$ gives $(s,j+1)$ entry of $E_{\mathbf b}$, and the coefficient in the middle is exactly 
 	\[
 	(C_f)_{i+1,j+1}=\alpha_{i+j}\binom{i+j}{i}.
 	\]
 	The desired factorization 
 	\[
 	A_{\mathbf a,\mathbf b}(f(x+y))
 	= E_{\mathbf a} C_f E_{\mathbf b}^T 
 	=
 	V_{\mathbf a}^{(k)}  (R C_f)  \bigl(W_{\mathbf b}^{(k)}\bigr)^T
 	\]  is concluded. 
 	It follows that 
 	\begin{align} \label{eq:lin}
 		\det (A_{\mathbf a,\mathbf b}(f(x+y)))
 		=
 		\det\bigl(V_{\mathbf a}^{(k)}\bigr) \det(R C_f) \det\bigl(W_{\mathbf b}^{(k)}\bigr).
 	\end{align}
 	
 	We now compute each factor: Since $V_{\mathbf a}^{(k)}$ lists descending powers, reversing columns to the standard Vandermonde order needs $\binom{n}{2}$ adjacent swaps:
 	\[
 	\det(V_{\mathbf a}^{(k)}) = (-1)^{\binom{n}{2}} \prod_{1\le r<r'\le n}(a_{r'}-a_r).
 	\]
 	This is the usual Vandermonde determinant:
 	\[
 	\det(W_{\mathbf b}^{(k)}) = \prod_{1\le s<s'\le n}(b_{s'}-b_s).
 	\]
 	We have $\det(R)=(-1)^{\binom{n}{2}}$. A standard Pascal-basis argument  gives
 	$
 	\det(C_f)=\alpha_k^{ n} (-1)^{\binom{n}{2}}
 	\prod\limits_{i=0}^k \binom{k}{i}.
 	$
 	Thus $\det(R C_f) = \det(R)\det( C_f)=\alpha_k^{ n}  
 	\prod\limits_{i=0}^k \binom{k}{i}$.

 	Multiplying the three determinants,
 	\begin{align*}
 		\det (A_{\mathbf a,\mathbf b}(f(x+y)))
 		&=
 		\Bigl[(-1)^{\binom{n}{2}}\prod_{r<r'}(a_{r'}-a_r)\Bigr] 
 		\Bigl[\alpha_k^{ n}  
 		\prod\limits_{i=0}^k \binom{k}{i}\Bigr] 
 		\Bigl[\prod_{s<s'}(b_{s'}-b_s)\Bigr] \\ &
 		=
 		\alpha_k^{ n} (-1)^{\binom{n}{2}}
 		\prod\limits_{i=0}^k \binom{k}{i} \prod_{1\leq r<r' \leq n}(a_{r'}-a_r) \prod_{1\leq s<s' \leq n}(b_{s'}-b_s),\end{align*}
 	as desired. 
 \end{proof}
 
 From Proposition \ref{prop:resultant-form},  for $f(t)=\sum\limits_{i=0}^k \alpha_i t^i$ with leading coefficient $\alpha_k\neq 0$ and $n=k+1$,   it is not difficult to see that $\det (A_{\mathbf a,\mathbf b}(f(x+y)))$ is independent from $\alpha_i$ for all $i\in\{0,1,\dots,k-1\}$.  Hence, we have the following corollary.

 \begin{corollary}
 	Let 
 	$n=k+1$  be integers and let $\mathbf a=(a_1,\dots,a_n)$  and  $\mathbf b=(b_1,\dots,b_n)$ be in $\mathbb{C}^n$.
 	If $f(t)=\sum\limits_{i=0}^k \alpha_i t^i$ with leading coefficient $\alpha_k\neq 0$ and $g(x)=\alpha_k t^k$, then \[\det (A_{\mathbf a,\mathbf b}(f(x+y)))=\det (A_{\mathbf a,\mathbf b}(g(x+y)))\]
 \end{corollary}
 
 \begin{remark} We note that  $p(x,y)=f(x+y)$ is homogeneous if and only if $f(t)= ct^k$  for some integer $k\geq 0$ and scalar $c$.  In this case,  the result in Proposition \ref{prop:resultant-form} agrees with the ones in Corollary~\ref{thm:sum-powers} and Remark~\ref{rem:2.1}.
 	
 \end{remark}

 \subsection{Equivariance under Linear Changes of Variables}\label{subsec:equivariance}

 We next study how $A_{\mathbf a,\mathbf b}(p(x,y))$ transforms when we apply a linear
 change of variables $(x,y)\mapsto(x',y')=B\cdot(x,y)$ with
 $B=\begin{psmallmatrix}\alpha&\beta\\ \gamma&\delta\end{psmallmatrix}\in\mathrm{GL}_2(\mathbb C)$.
 For a homogeneous $p(x,y)$ of degree $k$ put $p_B(x,y):=p(x',y')$.
 For the \emph{sum-form} $p(x,y)=f(x+y)$ we obtain a clean and explicit rule.
 
 \begin{proposition}\label{prop:GL2-sum}
 	Let $f(x)$ be a polynomial of degree $k$ and let   $p(x,y)=f(x+y)$. Let $n\le k+1$ and let 
 	\[
 	c:=\alpha+\gamma \text{~ and ~} d:=\beta+\delta.
 	\] Let   $\mathbf a=(a_1,\dots,a_n)$ and  $\mathbf b=(b_1,\dots,b_n)$ be vectors in $\mathbb{C}^n$.
 	Then 
 	\[
 	A_{\mathbf a,\mathbf b}(p_{B}(x,y)) = \left[f(ca_r+db_s)\right]_{r,s=1}^n
 	= A_{c\mathbf a,\ d\mathbf b}(f(x+y)).
 	\]
 	In particular, if $n=k+1$, then 
 	\[
 	\det (A_{\mathbf a,\mathbf b}(p_{B}(x,y)))
 	=  c^{\binom{n}{2}}  d^{\binom{n}{2}}  \det (A_{\mathbf a,\mathbf b}(p(x,y)))
 	.
 	\] 
 \end{proposition}
 
 \begin{proof} Since $p_{B}(x,y)=f(x'+y')=f\big((\alpha+\gamma)x+(\beta+\delta)y\big)=f(cx+dy)$, we have \[ (A_{\mathbf a,\mathbf b}(p_{B}(x,y)))_{rs} =f(ca_r+db_s)=(A_{c\mathbf a,\ d\mathbf b}(f(x+y)))_{rs}. \]  Precisely, $A_{\mathbf a,\mathbf b}(p_B (x,y))=A_{c\mathbf a,d\mathbf b}(f(x+y))$. 
 	
 	For the determinant scaling at $n=k+1$, use \eqref{eq:lin} in Proposition \ref{prop:resultant-form},  we have

 	\begin{align*}
 		\det (A_{\mathbf a,\mathbf b}(p_B(x,y))) =\det( A_{c\mathbf a,d\mathbf b}(f(x+y)))
 		=\det\bigl(V_{c\mathbf a}^{(k)}\bigr) \det(R C_f) \det\bigl(W_{d\mathbf b}^{(k)}\bigr).
 	\end{align*}
 	Since \[
 	\det(V_{c\mathbf a}^{(k)}) = c^{\binom{n}{2}} \det(V_{\mathbf a}^{(k)}) 
 	\text{ ~ and ~ }
 	\det(W_{d\mathbf b}^{(k)}) = d^{\binom{n}{2}} \det(W_{\mathbf b}^{(k)}),
 	\]
 	it can be deduced that 
 	\begin{align*}
 		\det (A_{\mathbf a,\mathbf b}(p_B(x,y)))  
 		=c^{\binom{n}{2}}d^{\binom{n}{2}}
 		\det (A_{\mathbf a,\mathbf b}\bigl(f(x+y)\bigr)).
 	\end{align*}
 \end{proof}
 
 Combining Proposition   \ref{prop:resultant-form} and Proposition  \ref{prop:GL2-sum}, an explicit formula for  $\det (A_{\mathbf a,\mathbf b}(f(x+y))$ follows immediately.

 \subsection{Behavior over Finite Fields}
 
 Next, we consider the behavior of $\det (A_{\mathbf a,\mathbf b}(p(x,y)))$ in the case where the vectors  $a_i,b_j$ are chosen randomly from  a finite field $\mathbb{F}_q$.  
 We  focus on the case where $n\leq k+1$ and  give an upper bound for the case where the determinant is identically zero.  This is obtained via an application of the Schwartz–Zippel lemma.

 \begin{lemma} [{\cite[Lemma A.36]{AroraBarak2009}, Schwartz–Zippel}] \label{lem:SZ}
 	Let $\mathbb{F}_q$ be a field of order $q$ and let $p(x_1,\dots,x_n)\in \mathbb{F}_q[x_1,\dots,x_n]$ be a nonzero polynomial of total degree $d\ge 0$. 
 	Let $S\subseteq \mathbb{F}_q$ be a  non-empty set and choose $(r_1,\dots,r_n)\in S^n$ uniformly at random. Then the probability   
 	\[
 	\Pr\bigl[p(r_1,\dots,r_n)=0\bigr]  \le  \frac{d}{|S|}.
 	\] 
 \end{lemma}

 \begin{proposition} \label{prop:FF-nonvanish} Let $\mathbb{F}_q$ be a finite field of order $q$ and let 
 	$n\le k+1$ be integers. 
 	Let  $\alpha_i\in\mathbb{F}_q \setminus \{0\}$ for $0\le i\le k$ and  let 
 	$\boldsymbol{a} =(a_1,\dots,a_n)$ and  $\boldsymbol{b} =(b_1,\dots,b_n) $  be independently and uniformly from $\mathbb{F}_q^n$.
 	If $p(x,y)=\sum\limits_{i=0}^k \alpha_i (x+y)^i$, then the probability    that  $\det (A_{\mathbf a,\mathbf b}(p(x,y)))=0$  is
 	\[
 	\Pr \left(\det (A_{\mathbf a,\mathbf b}(p(x,y)))=0\right) \le \frac{nk}{q}.
 	\]
 \end{proposition}
 
 \begin{proof} From \eqref{eq:lin}, we have 
 	
 	\begin{align*}  
 		\det (A_{\mathbf a,\mathbf b}(p(x,y)))
 		=
 		\det\bigl(V_{\mathbf a}^{(k)}\bigr) \det(R C_p) \det\bigl(W_{\mathbf b}^{(k)}\bigr) \text{ and }  \det(RC_p)=\alpha_k^{ n}  
 		\prod\limits_{i=0}^k \binom{k}{i}.
 	\end{align*}
 	Since $\alpha_i\ne 0 $ for all $1\leq i \leq n$,   $\det(RC_p)\ne 0$  which implies that  $\det (A_{\mathbf a,\mathbf b}(p(x,y)))$ is a nonzero polynomial.

 	For each  $1\leq r\leq n$ and $1\leq s \leq n$, we note that each entry $p(a_r,b_s)=\sum\limits_{i=0}^k \alpha_i(a_r+b_s)^i$ has degree at most $k$ in the
 	$2n$ variables $a_r$ and $b_s$. Using the Leibniz expansion,   $\det (A_{\mathbf a,\mathbf b}(p(x,y)))$ is a product
 	of $n$ entries and  each depending on a {distinct} pair $(a_r,b_s)$ which implies that   the total
 	degree is at most $nk$. Hence,     $\det (A_{\mathbf a,\mathbf b}(p(x,y)))$ is a nonzero polynomial of total degree at most  $ nk$ in $2n$ independent
 	uniform variables over $\mathbb{F}_q$. By  Lemma \ref{lem:SZ}, it follows that
 	\[
 	\Pr\left(\det (A_{\mathbf a,\mathbf b}(p(x,y)))=0\right)\le  \frac{nk}{q}.
 	\]
 	The proof is completed.
 \end{proof}

 \section{Conclusion and Remarks}\label{sec4} 
 Using the simple yet powerful Vandermonde factorization  
 $
 A_{\mathbf a,\mathbf b}(p(x,y))=V_{\mathbf a}^{(k)} D_\alpha (W_{\mathbf b}^{(k)})^T,
 $
 we have developed a unified framework for understanding determinants of polynomial evaluation matrices associated with homogeneous bivariate polynomials. Our analysis covers the full range of matrix sizes, from \(n\le k+1\) to the square borderline case \(n=k+1\) and beyond, where rank constraints force the determinant to be zero for \(n\ge k+2\).  
 
 In the square case \(n=k+1\), we obtained a closed-form expression for \(\det (A_{\mathbf a,\mathbf b}(p(x,y)))\) in terms of Vandermonde products and coefficient factors, recovering and extending classical formulas for matrices such as \([ (a_i+b_j)^k]\). For \(n\le k+1\), applying the Cauchy--Binet formula together with an elementary evaluation of generalized Vandermonde minors via complete homogeneous polynomials \(H_m\) yielded a symmetric expansion for the determinant, leading to immediate rank bounds and vanishing criteria.     
 
 Several applications illustrate the versatility of our determinant factorization.
 We derived a closed-form formula for the square sum-form case  where
 $p(x,y)=f(x+y)$ with 
 $n=k+1$.
 Beyond this special case, we established an explicit equivariance law under linear changes of variables for sum-form polynomials, and proved probabilistic non-vanishing bounds over finite fields via the Schwartz–Zippel lemma. 
 
 More generally, it would be of interest to analyze determinants of polynomial evaluation matrices associated with non-homogeneous bivariate polynomials.
 
 \section*{Acknowledgments}
 The authors would like to thank Natcha Karnjano 
 and Kullanut Ueaumakul for useful  discussion.  S. Jitman  was funded by the National Research Council of Thailand and Silpakorn University  under Research Grant  N42A650381.

\end{document}